\documentclass[11pt]{amsart}

\usepackage{amsfonts,epsfig}
\usepackage{latexsym}
\usepackage{amssymb}
\usepackage{amsmath}
\usepackage{amsthm}
\usepackage{graphics}
\usepackage[all]{xy}
\usepackage[T2A]{fontenc}
\usepackage{multirow}
\usepackage{hhline}
\usepackage{array, booktabs, ctable}
\usepackage{hyperref}
\usepackage{enumerate}

\usepackage{tabularray}

\usepackage{longtable}
\usepackage{slashbox}
\usepackage{booktabs}
\usepackage{diagbox}

\newtheorem{defn}{Definition}%[section]

\newtheorem{lemma}[defn]{Lemma}

\newtheorem{thm}[defn]{Theorem}

\newtheorem{prop}[defn]{Proposition}

\newtheorem{conj}[defn]{Conjecture}

\theoremstyle{definition}
\newtheorem*{remark}{Remark}

\newcommand{\Q}{\mathbb Q}

\newcommand{\Z}{\mathbb Z}

\newcommand{\QD}{\Q(\sqrt{D})}

\title[Squares in arithmetic progression over certain non-primitive quartic number fields]{Squares in arithmetic progression over certain non-primitive quartic number fields}

\author{Enrique Gonz\'alez--Jim\'enez}
\address{Universidad Aut{\'o}noma de Madrid, Departamento de Matem{\'a}ticas, Madrid, Spain}
\email{enrique.gonzalez.jimenez@uam.es}

\author{Nguyen Xuan Tho}
\address{Hanoi University of Science and Technology, Hanoi, Vietnam}
\email{tho.nguyenxuan1@hust.edu.vn} 

\thanks{The first author is supported by Grant PID2022-138916NB-I00 funded by MCIN/AEI/10.13039/501100011033 and by ERDF A way of making Europe. The second author is supported by the Vietnam National Foundation for Science and Technology Development (NAFOSTED) (grant number 101.04-2023.21).}

\subjclass{Primary: 11B25, 11G05; Secondary: 14G05.}
\keywords{Arithmetic progressions of squares, quadratic extensions, elliptic curves, rational points}

\begin{document}
\date{\today}

%\tableofcontents

\begin{abstract}
Let $D$ be a square-free integer. Under certain conditions on $D$, we characterize  non-constant arithmetic progressions of squares over quadratic extensions of $\Q(\sqrt{D})$. 
\end{abstract}
\maketitle
\section{Introduction}
Fermat stated that there are no non-constant arithmetic progressions of four squares over $\mathbb{Q}$ in 1640 and Euler proved this in 1780. It is natural to ask for what would happen in number fields. Over quadratic fields, Xarles \cite{art17} showed that there are no non-constant arithmetic progressions of six squares and Gonz\'alez-Jim\'enez and Xarles \cite{GJX} characterized non-constant arithmetic progressions of five squares. Over cubic fields, Bremner and Siksek \cite{art5} showed that there are no non-constant arithmetic progressions of five squares. 

In this paper, we extend the results of  Xarles \cite{art17} and of Gonz\'alez-Jim\'enez and Xarles \cite{GJX}  to quadratic extensions of $\Q(\sqrt{D})$, where $D$ is a square-free integer satisfying certain conditions on the $D$-quadratic twist of the elliptic curves
$$
E_0\,:\, y^2 = x^3 - x^2 - 9x + 9 \qquad \mbox{and}\qquad E_1\,:\, y^2 = x^3 - x^2 + x.
$$
For any elliptic curve $E$ and integer $D$, we denote by $E^D$ the $D$-quadratic twist of $E$. 

{
Let $n$ be a positive integer and $K$ be a field. Let $a_1,\dots,a_n\in K$ such that $(a_1^2,a_2^2,\dots,a_n^2)$ is an arithmetic progression of length $n$. We say that this arithmetic progression is equivalent to $(s^2 a_1^2,s^2 a_2^2,\dots,s^2a_n^2)$ for any $s\in K^*$ and to the arithmetic progression $(a_n^2,\dots,a_2^2,a_1^2)$. We say that $(a_1^2,a_2^2,\dots,a_n^2)$ is properly defined over a number field $K$ if $a_1,\dots,a_n\in K$  and  $\{a_1,\dots,a_n\}\not\subset F$ for any proper subfield $F$ of $K$.} Let $i=\sqrt{-1}$.

\

The main theorems of the paper are the following.

\begin{thm}\label{main}
Let $D$ be a square-free integer, $D\ne -1,\pm 2,3$, such that $\operatorname{rank}_{\Z}E_1^{\pm D}(\Q)= 0$ and $K$ a quadratic extension of $\QD$. Then
\begin{itemize}
\item[(A)] If $\operatorname{rank}_{\Z}E_0^{D}(\Q)=0$, then there does not exist any non-constant arithmetic progression of five squares properly defined over $K$.

\item[(B)] If $\operatorname{rank}_{\Z}E_0^{D}(\Q)\ne 0$ and the class number of $\QD$ is $1$, then a non-constant arithmetic progression of five squares properly defined over $K$ is, up to equivalence, of the form $\left(a^2,b^2,c^2,\alpha \, d^2,e^2\right)$, where $a,b,c,d,e,\alpha \in\Q(\sqrt{D})$ and $\alpha$ is non-square.
\end{itemize}
\end{thm}

\begin{thm}\label{maincor}
There exists a non-constant arithmetic progression of five squares properly defined over a quadratic extension $K$ of
\begin{itemize}
    \item $\Q(i)$ if and only if $K=\Q(i,\sqrt{2})$. In this case, $(-2,-1, 0,1,2)$ is the unique, up to equivalence, non-constant arithmetic progression of five squares properly defined over $K$.
     \item $\Q(\sqrt{3})$ if and only if $K=\Q(\sqrt{3},\sqrt{2})$. In this case, $(0,1,2,3,4)$ is the unique, up to equivalence, non-constant arithmetic progression of five squares properly defined over $K$.
     \item $\Q(\sqrt{2})$ if and only if $K=\Q(\sqrt{2},\sqrt{3})$, $K=\Q(\sqrt{2},i)$ or $K=\Q(\sqrt{2},\sqrt{\alpha})$, for some squarefree $\alpha \in\Q(\sqrt{2})$. In the first case, $(0,1,2,3,4)$ is the unique, up to equivalence, non-constant arithmetic progression of five squares properly defined over $\Q(\sqrt{2},\sqrt{3})$; and in the second case, $(-2,-1, 0,1,2)$ is the unique, up to equivalence, non-constant arithmetic progression of five squares properly defined over $\Q(\sqrt{2},i)$. In the last case, a non-constant arithmetic progression of five squares properly defined over $K$ is, up to equivalence, of the form $\left(a^2,b^2,c^2,\alpha \, d^2,e^2\right)$ where $a,b,c,d,e \in\Q(\sqrt{2})$. 
     \item $\Q(\sqrt{-2})$ if and only if $K=\Q(\sqrt{-2},i)$ or $K=\Q(\sqrt{-2},\sqrt{\alpha})$, for some squarefree $\alpha \in\Q(\sqrt{-2})$. In the first case, $(-2,-1, 0,1,2)$ is the unique, up to equivalence, non-constant arithmetic progression of five squares properly defined over $\Q(\sqrt{-2},i)$. In the second case, a non-constant arithmetic progression of five squares properly defined over $K$ is, up to equivalence, of the form $\left(a^2,b^2,c^2,\alpha \, d^2,e^2\right)$ where $a,b,c,d,e \in\Q(\sqrt{-2})$.
\end{itemize}  
\end{thm}
\begin{remark}
Note that $\operatorname{rank}_{\Z}E_1^{\pm D}(\Q)= 0$ and $\operatorname{rank}_{\Z}E_0^{D}(\Q)= 0$ if $D\in\{-1,\pm 2,3\}$.
\end{remark}

\begin{thm}\label{six}
Let $D$ be a square-free integer such that 
$$
\operatorname{rank}_{\Z}E_1^{\pm D}(\Q)= 0\,,
$$
and $\operatorname{rank}_{\Z}E_0^{D}(\Q)=0$ or if $\operatorname{rank}_{\Z}E_0^{D}(\Q)\ne 0$  the class number of $\QD$ is $1$. Then there does not exist any non-constant arithmetic progression of six squares properly defined over any quadratic extension of $\QD$. 
\end{thm}

\begin{remark}
The following list shows all the square-free integers $D$, $|D|<200$, satisfying $\operatorname{rank}_{\Z}E_1^{\pm D}(\Q)= 0$ and  $\operatorname{rank}_{\Z}E_0^{D}(\Q)= 0$\,:
$$
-158, -123, -62, -51, -3, -2, -1, 2, 3, 7, 31, 51, 62, 79, 103, 110, 123, 127,
151, 158, 194, 195, 199.
$$
If we allow $\operatorname{rank}_{\Z}E_0^{D}(\Q)\ne 0$ and the class number of $\QD$ is $1$ then $D\in\{-7, 38, 86\}$. 
\end{remark}

\begin{remark}
There exist square-free integers $D$ and non-constant arithmetic progressions of six squares properly defined over a quadratic extension of $\QD$. For example, when $D=409$ then $$(7^2, 13^2, 17^2, (\sqrt{409})^2, 23^2, (\sqrt{649})^2)$$ is an arithmetic progression properly defined over $\Q(\sqrt{409},\sqrt{649})$.
\end{remark}

\begin{conj}\label{Conj}
There are infinitely many non-constant arithmetic progressions of five squares pro\-perly defined over a quadratic extension $K_1$ (resp. $K_2$) of $\Q(\sqrt{-2})$ (resp. of $\Q(\sqrt{2}))$.
\begin{itemize}
\item[(i)] In the case $\Q(\sqrt{-2})$, any of them are, up to equivalence, of the form $\left(a^2,b^2,-2c^2,-m d^2,-2e^2\right)$, where $a,b,c,d,e,m\in\Z$ and $m>0$ is square-free. In particular, $K_1=\Q(\sqrt{-2},\sqrt{-m})$. 
\item[(ii)] In the case $\Q(\sqrt{2})$, any of them are, up to equivalence, of the form $\left(2 a^2,b^2,2c^2,m,e^2\right)$, where $a,b,c,e,m\in\Z$ and $m$ is square-free. In particular, $K_2=\Q(\sqrt{2},\sqrt{m})$. 
\end{itemize}
\end{conj}

Conjecture \ref{Conj} is supported by the calculation in Section \ref{sect_para}. For the cases $D\in\{-7, 38, 86\}$, there is not a similar conjecture.

Our approach is different from the approach in \cite{GJX} and \cite{art17}. We follow the approach in Mordell's paper \cite{art12}, where he gave alternative proofs to the results of Aigner \cite{art1} and Faddeev \cite{art7} on the solution of the equation $x^4+y^4=1$ in quadratic and cubic number fields. Mordell's approach has the advantage that it is very concrete in calculation.

Another important ingredient  is the theory concerning the growth of the torsion subgroup of elliptic curves under base change. Upon assuming $\operatorname{rank}_{\Z}E_1^{\pm D}(\Q)= 0$, it becomes imperative to determine the torsion subgroup of certain elliptic curves defined over $\Q$ over any quadratic field. The answer to this inquiry lies within the LMFDB online database  \cite{lmfdb}. The theory of torsion growth has undergone extensive scrutiny in recent years. Particularly noteworthy references are \cite{GJN1,GJN2} and \cite{GJT1,GJT2,N} (for quadratic fields). 

In the following subsection we include all the elliptic curves that are relevant to this article, together with the necessary information concerning rational points over quadratic fields. This material will be used in the proofs of the previous theorems.

\subsection*{Data of some elliptic curves}
Table \ref{TableC} shows the elliptic curves used in this article. For each curve \( C_k \), with \( k = 0,1,\dots,6 \), the third column gives the corresponding label in the LMFDB database \cite{lmfdb}, while the fourth column lists the corresponding \(\mathbb{Q}\)-isomorphic elliptic curve.

\begin{center}
\renewcommand{\arraystretch}{1.3}
\begin{longtable}{|l|l|l|l|}
\caption{Curves} \label{TableC} \\
\hline 
$C_0$ & $y^2 = x^4-2x^3+2x^2+2x+1$  &  \href{https://www.lmfdb.org/EllipticCurve/Q/192a2/}{\texttt{192.a2}} &  $E_0$\\
\hline 
$C_1$& $y^2 = (x+4)(x^2+4x+16)$  & \href{https://www.lmfdb.org/EllipticCurve/Q/24a4/}{\texttt{24.a5}} & $E_1$\\
\hline
$C_2$& $y^2 = x(x^2+4x+16)$  & \href{https://www.lmfdb.org/EllipticCurve/Q/48a4/}{\texttt{48.a5}}&  $E^{-1}_1$\\
\hline 
$C_3$& $y^2 = x(x+4)(x^2+4x+16)$ &  \href{https://www.lmfdb.org/EllipticCurve/Q/24a4/}{\texttt{24.a5}}&  $E_1$\\
\hline 
$C_4$& $y^2 = x(x^2+14x+1)$ &  \href{https://www.lmfdb.org/EllipticCurve/Q/24a3/}{\texttt{24.a2}}&  $E_4$\\
\hline 
$C_5$ & $y^2 = x^4+4x^2+16$ &  \href{https://www.lmfdb.org/EllipticCurve/Q/48a1/}{\texttt{48.a4}} &  $E^{-1}_6$\\
\hline 
$C_6$ & $y^2 = x^4+14x^2+1$  &  \href{https://www.lmfdb.org/EllipticCurve/Q/24a1/}{\texttt{24.a4}} &  $E_6$\\
\hline
\end{longtable}
\end{center}
Note that the elliptic curves $E_1,E_4$ and $E_6$ belongs to the same $\mathbb Q$-isogeny class.  In particular, for any integer $D$, the ranks of any $D$-quadratic twist coincide. In particular, we have $\operatorname{rank}_{\Z}C_k(\Q)=0$ for $k=1,\dots,6$. Let $E$ be an elliptic curve defined over $\Q$. Using the known formula (see \cite{Kra})
$$
\operatorname{rank}_{\Z}E(\QD)=\operatorname{rank}_{\Z}E(\Q)+
\operatorname{rank}_{\Z}E^D(\Q),
$$
we obtain that $\operatorname{rank}_{\Z}E_1^{\pm D}(\Q)= 0$ is equivalent to  $\operatorname{rank}_{\Z}C_k(\QD)=0$ for $k=1,\dots,6$. Assuming $\operatorname{rank}_{\Z}E_1^{\pm D}(\Q)= 0$ for the square-free integer $D$, to compute $C_k(\QD)$, for $k=1,\dots,6$, we need to compute the torsion subgroup over $\QD$. In our particular case, given an elliptic curve $E$ defined over $\Q$, there are only a finite number of square-free integers $D$ such that $E(\Q)_{\operatorname{tors}}\ne E(\QD)_{\operatorname{tors}}$. If the conductor of $E$ is less than $400000$, these square-free integers appear in the LMFDB database \cite{lmfdb} and are listed in Table \ref{growth}.
\begin{center}
\begin{longtable}{|c|c|}
\caption{Torsion growth over quadratic fields}\label{growth}\\
\hline 
$E$ & \mbox{$D$ such that $E(\Q)_{\operatorname{tors}}\ne 
E(\QD)_{\operatorname{tors}}$}\\
\hline 
 
$E_0$ & $-2$ \\
\hline
$E_1$   &  $-1$, $\pm 3$ \\
\hline
$E^{-1}_1$     &  $-1$,$\pm 3$\\
\hline 
$E_4$ & 2,3,6\\
\hline
$E_6$ & $\emptyset$ \\
\hline 
$E_6^{-1}$ & $-1$,$\pm 3$\\
\hline
\end{longtable}
\end{center}
Note that for those elliptic curves and values of $D$, we have $\operatorname{rank}_{\Z}E^D(\Q)=0$ except in the case $E^6_4$. Finally, for any square-free integer $D$, Tables \ref{points} and \ref{pointsC4}  show  $C_k(\QD)$ for $k=1,\dots,6$. When $D=-1,2,\pm 3$, the points at the second column correspond to $C_k(\Q)$ and they must be added to the corresponding set $C_k(\QD)$.

\renewcommand{\arraystretch}{1.3}
\begin{table}[htbp] 
\begin{center} \caption{$C_k(\QD)$ for $k=1,2,3,5,6$.}\label{points}
\begin{tabular}{|c|c|c|c|c|}
\cline{2-5} 
 \multicolumn{1}{c|}{}
%----------------------------------------------
 & $D\ne -1,\pm 3$ & $D=-1$ & $D=-3$ & $D=3$  \\ 
\hline
$C_1$ & $\begin{array}{c}(-4,0)\\ (0,\pm 8) \end{array}$& $(-4\pm 4i,\pm8)$ & $\begin{array}{c}(-8,\pm8\sqrt{-3})\\(-2\pm 2 \sqrt{-3},0)\end{array}$ & $\begin{array}{c}
(4-4\sqrt{3},\pm(24-16\sqrt{3}))\\ 
(4+4 \sqrt{3},\pm(24+16\sqrt{3}))
\end{array}$  \\
\hline
$C_2$ & $(0,0)$ & $\begin{array}{c}(\pm 4i,\pm 8i)\\(-4,\pm 8i) \end{array}$& $(-2\pm 2\sqrt{-3},0)$ & $(4,\pm 8\sqrt{3}) $ \\ 
\hline
$C_3$ & $\begin{array}{c} (0,0)\\(-4,0) \end{array}$& $(-2\pm 2 i,\pm 8i)$ & $\begin{array}{c}(-2\pm 2\sqrt{-3},0)\\(-2,\pm 4 \sqrt{-3})\end{array}$ & $(-2\pm 2 \sqrt{3},\pm 8\sqrt{3})$  \\
\hline
$C_5$ & $(0,\pm 4)$ & $(\pm 2i,\pm 4)$ & $(\pm1\pm\sqrt{-3},0)$ & $(\pm 2,\pm 4\sqrt{3})$  \\ 
\hline
$C_6$ & $\begin{array}{c} (0,\pm 1)\\ (\pm 1,\pm 4) \end{array}$ & $-$ & $-$ & $- $ \\ 
\hline
\end{tabular} 
\end{center}
\end{table}

\renewcommand{\arraystretch}{1.3}
\begin{table}[htbp] 
\begin{center} \caption{$C_4(\QD)$}\label{pointsC4}
\begin{tabular}{|c|c|c|c|}
\cline{2-4} 
 \multicolumn{1}{c|}{} & $D\ne 2,3,6$ & $D=3$ & $D=2$  \\ 
\hline
$C_4$ & $\begin{array}{c} (0,0)\\ (1,\pm 4)  \end{array}$ & $\begin{array}{c}(-1,-2\sqrt{3})\\ (7\pm 4\sqrt{3},0)\end{array}$ & $\begin{array}{c}
(-3-2\sqrt{2},\pm(8+6\sqrt{2}))\\ 
(-3+2\sqrt{2},\pm(8-6\sqrt{2}))\\ 
\end{array}$\\
   \hline
\end{tabular} 
\end{center}
\end{table}

\newpage

\begin{lemma}\label{newlemma}
Let $D$ be a square-free integer such that $\operatorname{rank}_{\Z}E_0^{D}(\Q)=0$. Then $E_0(\Q(\sqrt{D}))=E_0(\Q)$, except when $D=\pm 2$.
\end{lemma}

\begin{proof}
Table~\ref{growth} shows that the torsion subgroup grows only when $D=-2$. Thus, for $D\neq -2$, we have $E_0(\Q(\sqrt{D}))_{\operatorname{tors}}=E_0(\Q)_{\operatorname{tors}}$. Therefore, if $D\neq -2$ and $\operatorname{rank}_{\Z}E_0^{D}(\Q)=0$, we deduce that $E_0(\Q(\sqrt{D}))$ is isomorphic to $E_0(\Q)$. Let us prove that the only case where $E_0(\Q(\sqrt{D}))\neq E_0(\Q)$ under these assumptions occurs when $D=2$. Assume that there exists $P\in E_0(\Q)$ and $R\in E_0(\Q(\sqrt{D}))\setminus E_0(\Q)$ such that $P=nR$ for some $n\in \Z$. Let $\sigma\in \operatorname{Gal}(\Q(\sqrt{D})/\Q)$ be the nontrivial automorphism. Then $P=n\,\sigma(R)$, and hence $n(R-\sigma(R))=\mathcal{O}$. This implies that $R-\sigma(R)\in E_0(\Q)[n]$. Since $E_0(\Q)_{\operatorname{tors}}=E_0[2]$, we conclude that $n=2$. 

Now recall that $E_0(\Q)=\langle T_1 \rangle \oplus \langle T_2 \rangle \oplus \langle P_0 \rangle$, where $T_1=(-3,0)$, $T_2=(1,0)$, and $P_0=(5,8)$. Thus any $P\in E_0(\Q)$ can be written as $ P=n_1T_1+n_2T_2+mP_0$ for some $n_1,n_2\in \{0,1\}$ and $m\in \Z$. It is therefore enough to check the existence of such $R$ when $P=n_1T_1+n_2T_2+P_0$ for some $n_1,n_2\in \{0,1\}$. Let us prove the above fact: suppose first that $m=2n$ for some $n\neq 0$. Then we would obtain $2(R+nP_0)=n_1T_1+n_2T_2$,  which is impossible since the left-hand side has infinite order while the right-hand side has order  { dividing} $2$. Suppose instead that $m=2n+1$ for some $n\neq 0$. Then we obtain $2(R+nP_0)=n_1T_1+n_2T_2+P_0$, which gives the desired relation.  

{ Finally, one computes that the only solutions to $2 R=n_1T_1+n_2T_2+P_0$, $n_1,n_2\in \{0,1\}$, are given by $n_1=n_2=0$ and $R=(1+2\sqrt{2},-4)+T$ with $T\in E[2]$.}  This completes the proof.
\end{proof}

{\bf Acknowledgments.} Nguyen Xuan Tho is also supported by Vietnam Institute for Advanced Study in Mathematics (VIASM) from April 2025 to May 2025. The author really appreciates the Institute for their help and funding. We would like to thank Xavier Xarles for pointing out that the elliptic curves $E_1,E_4$ and $E_6$ belongs to the same $\mathbb Q$-isogeny class. The authors would like to thank the referees for their careful reading of the manuscript and for their valuable suggestions. In particular, we are grateful to the last referee, whose extensive and detailed comments greatly enhanced the overall quality of the paper. We are sincerely thankful for their contributions.

\

All computations in this paper are done in \texttt{Magma} \cite{Magma} and \texttt{Mathematica} \cite{mathematica}. 

\section{Setting}
Let $D$ be a square-free integer and $K$ be a quadratic extension of $\mathbb{Q}(\sqrt{D})$. If $S\subseteq K$, we denote by $S^2$ to the set of squares in $S$. Let  $a, b, c, d, e \in K$ be such that $\left(a^{2}, b^{2}, c^{2}, d^{2}, e^{2}\right)$ is an arithmetic progression. Then $[a: b: c: d: e]\in V(K) \subset \mathbb{P}^4(K)$, where $V$ is the genus 5 curve defined by the system of equations
$$
V\,\,:\,\, \{a^{2}+e^{2}=2 c^{2}\,,\,\, a^{2}+c^{2}=2 b^{2}\,,\,\,c^{2}+e^{2}=2 d^{2} \}.
$$
Equation $a^2+e^2=2c^2$ has the parametrization
\begin{equation}\label{eq2}
[a: c: e]=\left[t^{2}-2 t-1: t^{2}+1: t^{2}+2 t-1\right] 
\end{equation}
with $t \in K$. Applying  \eqref{eq2} in the equations of $V$ gives, up to sign,
\begin{equation}\label{eq3}
\left\{
\begin{array}{l}
a = t^{2}-2 t-1, \\
b= \sqrt{t^{4}-2 t^{3}+2 t^{2}+2 t+1}, \\
c= t^{2}+1, \\
d= \sqrt{t^{4}+2 t^{3}+2 t^{2}-2 t+1}, \\
e= t^{2}+2 t-1.
\end{array}
\right.
\end{equation}
Replacing $t$ by $-t$ in \eqref{eq3} gives the arithmetic progression $(e^2,d^2,c^2,b^2,a^2)$, equivalent to $(a^2,b^2,c^2,d^2,e^2)$, and the constant arithmetic progressions correspond to the cases $t=0,\pm 1$.

Let $G(x)=x^{4}-2 x^{3}+2 x^{2}+2 x+1$. By \eqref{eq3}, we have 
\begin{equation}\label{eq6}
G(\pm t)\in {K}^2.
\end{equation}
Assume $t\ne 0$. Then 
\begin{equation}\label{eq7}
\frac{G(\pm t)}{t^2}= t^{2}+\frac{1}{t^{2}}\mp 2\left(t-\frac{1}{t}\right)+2=s^2\mp 2s+4 \in {K}^{2},
\end{equation}
where $s=t-\frac{1}{t}$. Since $s^{2}+4=\left(t+\frac{1}{t}\right)^{2} \in K^{2}$, we have
\[
\begin{cases}\left(s^{2}+4\right)\left(s^{2}-2 s+4\right)\left(s^{2}+2 s+4\right) \in K^{2},\\ s^{2}\left(s^{2}-2 s+4\right)\left(s^{2}+2 s+4\right) \in K^{2}.
\end{cases}
\]
Let $r=s^{2}$. Then
\begin{equation}\label{eq8}
\begin{cases}
(r+4)\left(r^{2}+4 r+16\right) \in K^{2}, \\
r\left(r^{2}+4 r+16\right) \in K^{2}. 
\end{cases}
\end{equation}
We say that an arithmetic progression is elementary if it is equivalent to a constant arithmetic progression or if an element in the arithmetic progression is $0$. Table \ref{elementary} shows the corresponding values of $t,s$, and $r$ for elementary arithmetic progressions of five squares, where $i=\sqrt{-1}$. 
\begin{longtblr}
[caption = {Elementary arithmetic progressions}, label=elementary]
{cells = {mode=imath},hlines,vlines,colspec  = cccccc}
%----------------------------------------------
 t & s & r & \mbox{case} & (a^2,b^2,c^2,d^2,e^2) & [a: b: c: d: e]  \\
%----------------------------------------------
0 & - & - & \SetCell[r=2]{c} \text{constant} & (1,1,1,1,1) & [\pm 1:\pm 1: \pm 1 :\pm 1 : \pm 1]\\
%----------------------------------------------
\pm 1 & 0 & 0 &  & (a^2,a^2,a^2,a^2,a^2) & [\pm a:\pm a: \pm a :\pm a : \pm a]\\
%----------------------------------------------
\pm i & \pm 2 i & -4 & c=0 & (-2,-1,0,1,2) & [ \pm i \sqrt{2}: \pm i: 0: \pm 1: \pm \sqrt{2}] \\
%----------------------------------------------
1\pm \sqrt{2} & 2 & 4 & a=0 & (0,1,2,3,4) & [0:\pm 1:\pm \sqrt{2}:\pm \sqrt{3}:\pm 2]\\
%----------------------------------------------
-1\pm \sqrt{2} & -2 & 4 & e=0 & (4,3,2,1,0)& [\pm 2:\pm \sqrt{3} : \pm \sqrt{2} : \pm  1 : 0]\\
%----------------------------------------------
G(t)=0 & 1\pm\sqrt{-3} & -2\pm 2 \sqrt{-3} & b=0 & (-1,0,1,2,3)& [ \pm i: 0: \pm 1: \pm \sqrt{2}: \pm \sqrt{3}]\\
%----------------------------------------------
G(-t)=0 & -1\pm\sqrt{-3} & -2\mp 2 \sqrt{-3} & d=0 & (-3,-2,-1, 0,1)& [ \pm \sqrt{3}: \pm \sqrt{2} : \pm 1 : 0 :\pm i] %\\
%----------------------------------------------
\end{longtblr}
Note that the case $b=0$ or $d=0$ is not possible since  $\Q(a,b,c,d,e)=\Q(i,\sqrt{2},\sqrt{3})$ is a degree $8$ number field. 

\

The proof of Theorem \ref{main} and \ref{maincor} relies on the following result:

\begin{prop}\label{prop}
Let $D$ be a square-free integer and $K$ be a quadratic extension of $\mathbb{Q}(\sqrt{D})$. Let  $a, b, c, d, e \in K$ be such that $\left(a^{2}, b^{2}, c^{2}, d^{2}, e^{2}\right)$ is a non elementary arithmetic progression given by \eqref{eq3}. If $\operatorname{rank}_{\Z}E_1^{\pm D}(\Q)= 0\,,$ then
\begin{enumerate}
\item\label{prop1} $r \in \QD$,
\item\label{prop2} $s \in \QD$,
\item\label{prop3} $t\in\QD$,
\item\label{prop4}  $t\in\Q$, if $D\ne \pm 2$ and  $\operatorname{rank}_{\Z}E_0^{D}(\Q)=0$.
\end{enumerate}
\end{prop}

\begin{remark}
For any $t\in\Q$, the parametrization \eqref{eq3} gives the arithmetic progressions of five squares 
$$
((t^2-2t-1)^2,G(t),(t^2+1)^2,G(-t),(t^2+2t-1)^2),$$
that is, in general, properly defined over the biquadratic number field $\Q(\sqrt{G(t)},\sqrt{G(-t)}\,)$.
\end{remark} 
\section{Proof of Proposition \ref{prop}\,\eqref{prop1}}
Let \( t \in K \setminus \{0, \pm 1, \pm i, \pm 1 \pm \sqrt{2}\} \) corresponding to a non-elementary arithmetic progression, as described in ~\eqref{eq3}, and let \( r = (t - 1/t)^2 \). We show that $r \in \Q(\sqrt{D})$ under the hypothesis $\operatorname{rank}_{\Z}E_1^{\pm D}(\Q)=0$.

Assume $r \notin \Q(\sqrt{D})$. Then $K=\Q(\sqrt{D})(r)$. By \eqref{eq8}, there exist $\alpha_1,\alpha_2 ,\beta_1,\beta_2 \in \Q(\sqrt{D})$ such that
\begin{equation}\label{eq9-10}
\begin{array}{rcl}
(r+4)\left(r^{2}+4 r+16\right)&=&(\alpha_1+\beta_1 r)^{2}, \\
r\left(r^{2}+4 r+16\right)&=&(\alpha_2+\beta_2 r)^{2}. 
\end{array}
\end{equation}
Since $[K: \Q(\sqrt{D})]=2$, $r$ is a root of a quadratic irreducible polynomial $Q(x)$ in $\Q(\sqrt{D})[x]$. By \eqref{eq9-10}, there exist $r_{1}, r_{2} \in \Q(\sqrt{D})$ such that
\begin{equation}\label{eq11-12}
\begin{array}{rcl}
(x+4)\left(x^{2}+4 x+16\right)-(\alpha_1+\beta_1 x)^{2}&=&Q(x)\left(x-r_{1}\right),\\
x\left(x^{2}+4 x+16\right)-(\alpha_2+\beta_2 x)^{2}&=&Q(x)\left(x-r_{2}\right).
\end{array}
\end{equation}
Hence, $\left(r_{k}, \alpha_k+\beta_k r_{k}\right)\in C_k(\Q(\sqrt{D}))$ for $k=1,2$, where
$$
\begin{array}{l}
C_{1}\colon y^{2}=(x+4)\left(x^{2}+4 x+16\right),\\
C_{2}\colon y^{2}=x\left(x^{2}+4 x+16\right).     
\end{array}
$$
For  $P_1=(x_1,y_1)\in C_1(\Q(\sqrt{D}))$ and $P_2=(x_2,y_2)\in C_2(\Q(\sqrt{D}))$, we have
$$
\begin{array}{rcl}
(x+4)\left(x^{2}+4 x+16\right)-(y_1-\beta_1 x_1+\beta_1 x)^{2}&=&Q(x)\left(x-{x_{1}}\right),\\
x\left(x^{2}+4 x+16\right)-(y_2-\beta_2 x_2+\beta_2 x)^{2}&=&Q(x)\left(x-{ x_{2}}\right).
\end{array}
$$
Now, since $x-x_1$ and $x-x_2$ divide the polynomials on the left-hand side, performing the division explicitly yields the following 
$$
\begin{array}{l}
Q(x)=x^2+(8 - \beta_1^2 + x_1)x+32 + 8 x_1 + \beta_1^2 x_1 + x_1^2 - 2 \beta_1 y_1,\\
Q(x)=x^2+(4 - \beta_2^2 + x_2)x+16 + 4 x_2 + \beta_2^2 x_2 + x_2^2 - 2 \beta_2 y_2.
\end{array}
$$
Upon equating the two polynomials, we obtain the following system of equations:
 $$
\left\{
\begin{array}{l}
x_1 \left(\beta_1^2+x_1+8\right)-2 \beta_1 y_1-x_2 \left(\beta_2^2+x_2+4\right)+2 \beta_2 y_2+16=0,\\
\beta_1^2-\beta_2^2-x_1+x_2-4=0.
\end{array}
\right.
$$
We find $(\beta_1,\beta_2)$ by solving the above system of quadratic equations. For $k=1,2$, let $m_{\beta_k}(x)$ be the minimal polynomial of $\beta_k$ and let $d_k$ be the degree of $m_{\beta_k}(x)$. Given $D$, we give a table where for each point $P_1\in C_1(\Q(\sqrt{D}))$ and each point $P_2\in C_2(\Q(\sqrt{D}))$, we show $(\beta_1,\beta_2)$ and the corresponding quadratic polynomial $Q(x)$ if $d_1,d_2\leq 2$, otherwise we show $[d_1,d_2]$ or $\{\}$ if there does not exist solution $(\beta_1,\beta_2)$. 

For $k=1,2$, replacing $P_k$ by $-P_k$ is equivalent to changing $\beta_k$ by $-\beta_k$. Therefore, we only show one element in each set $\{\pm P_k\}$. The $\QD$-rational points in the curves $C_1$ and $C_2$ appear in the Table \ref{points}.\\
\begin{itemize}
    \item Case $D\ne -1,\pm 3$. Then 
\begin{center}
\begin{longtable}{|c|c|}
\caption{Case $D\neq -1,\pm 3$}
\label{Dneqminus1pm3}\\
\hline
%----------------------------------------------
\backslashbox{$P_1$}{$P_2$} & $(0,0)$ \\
\hline
%----------------------------------------------
$(-4,0)$ & $\begin{array}{c} (\beta_1,\beta_2)=(0,0) \\ Q(x)=x^2+4x+16 \end{array}$ \\
\hline
%----------------------------------------------
 $(0,8)$ &  $\begin{array}{c} (\beta_1,\beta_2)=(1,\pm \sqrt{-3}) \\ Q(x)=x^2+7x+16 \end{array}$ \\
 \hline
%----------------------------------------------
\end{longtable}
\end{center}  
 The case $(\beta_1,\beta_2)=(0,0)$ gives $Q(x)=x^2+4x+16$. Hence, $r^2+4r+16=Q(r)=0$. Thus $r=-2\pm 2\sqrt{-3}$, which is also not possible since it corresponds to elementary arithmetic progressions (see Table \ref{elementary}).
The case $(\beta_1,\beta_2)=(1,\pm \sqrt{-3})$ is not possible since $\sqrt{-3}\notin \QD$. \\
 \item Case $D=3$. Then
\begin{center}
\begin{longtable}{|c|c|c|}
\caption{Case $D=3$}
\label{Deq3}\\
\hline
\backslashbox{$P_1$}{$P_2$} & $(0,0)$ & $(4,8\sqrt{3})$\\
\hline 
%----------------------------------------------
$(-4,0)$ & $\begin{array}{c}(0,0) \\ x^2+4x+16 \end{array}$ & $[4,2]$ \\
\hline 
%----------------------------------------------
$(0,8)$ &  $\begin{array}{c} (1,\pm \sqrt{-3}) \\ x^2+7x+16 \end{array}$ & $[4,4]$ \\
\hline 
%----------------------------------------------
$(4\pm 4 \sqrt{3},24\pm 16\sqrt{3})$ & $[4,8]$& $[8,8]$\\
\hline 
\end{longtable}  
\end{center}
%----------------------------------------------
The proof is analogous to the case  $D\neq -1,\pm 3$.

\newpage

 \item Case $D=-1$. Then
 \begin{center}
\begin{longtable}{|c|c|c|c|}
\caption{Case $D=-1$}
\label{Deqminus1}\\
\hline 
\backslashbox{$P_1$}{$P_2$} & (0,0) & $(\pm 4i,8i)$ & $(-4,8i)$\\
\hline 
 $(-4,0)$ & $\begin{array}{c}(0,0) \\ x^2+4x+16 \end{array}$ & [8,4] & $\begin{array}{c}(\pm\sqrt{3},-i) \\ x^2+x+4 \end{array}$ \\
\hline 
%----------------------------------------------
 \multirow{5}{*}{$(0,8)$} &   \multirow{5}{*}{$\begin{array}{c} (1,\pm \sqrt{-3}) \\ x^2+7x+16 \end{array}$} &  \multirow{5}{*}{$[8,8]$} & $[2,4]$\\\cline{4-4} 
 & & & $(-2,2i)$ \\
 & &  & $x^2+4x+64$\\\cline{4-4}
 & & & $(2,-2i)$\\
 & & & $x^2+4x$\\\cline{4-4}
  \hline 
%----------------------------------------------
$(-4\pm 4i,8)$ & $[4,8]$ & $[4,4]$ & $[8,8]$\\
\hline 
%----------------------------------------------
\end{longtable}
\end{center}
The cases $(\beta_1,\beta_2)=(0,0)$ and $(\beta_1,\beta_2)=(1,\pm \sqrt{-3})$ are similar to the case $D\ne -1,\pm 3$. The case
$(\beta_1,\beta_2)=(\pm \sqrt{3},-i)$ is not possible since $\sqrt{3}\notin \Q(i)$. If $(\beta_1,\beta_2)=(-2,2i)$ then $Q(r)=r^2+4r+64=0$. Hence, $s^4+4s^2+64=0$, which is not possible since  the polynomial $x^4+4x^2+64$ is irreducible over $\Q(i)$. If $(\beta_1,\beta_2)=(2,-2i)$ then $Q(r)=r^2+4r=0$, which is not possible since $Q(x)$ must be irreducible over $\Q(i)$. \\ 

 \item Case $D=-3$. Then 
 
 \
 
\begin{center}
\begin{longtable}{|c|c|c|c|}
\caption{Case $D=-3$}\label{Deqm3}\\
\hline
\diagbox{$P_1$}{$P_2$} & (0,0) & $(-2+2\sqrt{-3},0)$ & $(-2-2\sqrt{-3},0)$\\
%----------------------------------------------
\hline 
$(-4,0)$ & $\begin{array}{c}(0,0) \\ x^2+4x+16 \end{array}$ & $[4,1]$ & $[4,1]$\\
%----------------------------------------------
\hline 
$(0,8)$ &  $\begin{array}{c} (1,\pm \sqrt{-3}) \\ x^2+7x+16 \end{array}$  &$[4,4]$ & $[4,4]$
\\
\hline 
\multirow{3}{*}{$(-8,8\sqrt{-3})$} & \multirow{3}{*}{$[4,4]$} & $[2,4]$ & $[2,4]$ \\\cline{3-4} 
& & $(1-\sqrt{-3},0)$ & $(-1-\sqrt{-3},0)$\\
& & $x(x+2+2\sqrt{-3})$ & $x(x+2-2\sqrt{-3})$\\
\hline 
$(-2+2\sqrt{-3})$ & $[2,4]$ & $\{\}$ & $[4,2]$ \\
%----------------------------------------------
\hline 
$(-2-2\sqrt{-3})$ & $[2,4]$  & $[4,2]$ & $\{\}$ \\
\hline 
%----------------------------------------------
\end{longtable}
\end{center}
The case $(\beta_1,\beta_2)=(0,0)$  is not possible by the same argument given in the previous cases. {If $(\beta_1,\beta_2) = (1, \pm \sqrt{-3})$, then $Q(x) = x^2 + 7x + 16$. Hence, $0 = Q(s^2) = (s^2 + s + 4)(s^2 - s + 4)$. We obtain $s = (\pm 1 \pm \sqrt{-15})/2$, which yields values of $t$ whose minimal polynomials are $x^4 \pm x^3 + 2x^2 \mp x + 1$. A computation in \texttt{Magma} shows that $\Q(t) = \Q(\sqrt{-3}, \sqrt{5})$, so that $[\Q(t) : \Q(\sqrt{-3})] = 2$. However, in these cases neither $G(t)$ nor $G(-t)$ is a square in $\Q(t)$. Therefore, these cases cannot occur.} If $(\beta_1,\beta_2)=(\pm 1-\sqrt{-3},0)$ then $Q(x)=x(x+2\pm 2\sqrt{-3})$, which is not possible since $Q(x)$ must be irreducible over $\Q(\sqrt{-3})$.
\end{itemize}

Hence, $r \in \Q(\sqrt{D})$.
\section{Proof of Proposition \ref{prop}\,\eqref{prop2}}
Let \( t \in K \setminus \{0, \pm 1, \pm i, \pm 1 \pm \sqrt{2}\} \) corresponding to a non-elementary arithmetic progression, as described in ~\eqref{eq3}, and let \( s = t - 1/t\). Note that in particular $s\not\in\{ 0,\pm 2 i,\pm 2\}$. We will show that $s \in \Q(\sqrt{D})$ under the hypothesis $
\operatorname{rank}_{\Z}E_1^{\pm D}(\Q)=0$.

Assume $s \notin \Q(\sqrt{D})$. Equivalently, $r=s^2\notin \QD^2$. Then $K=\Q(\sqrt{D})(s)$. Therefore, it follows from \eqref{eq7} that
$$
s^{2}\pm 2 s+4=(\mu\pm\psi s)^{2}
$$
with $\mu, \psi \in \Q(\sqrt{D})$. By Proposition \ref{prop}\,\eqref{prop1}, $s^2=r\in \Q(\sqrt{D})$. Hence,
\begin{equation}\label{eq21}
r^{2}+4 r+16=\left(s^{2}-2 s+4\right)\left(s^{2}+2 s+4\right)=\left(\mu^{2}-s^2 \psi^{2}\right)^{2}=\left(\mu^{2}-r \psi^{2}\right)^{2} \in \Q(\sqrt{D})^{2},
\end{equation}
 Since $t^{2}-s t-1=0$, we have $t=(s\pm \alpha)/2$, where $\alpha\in K$ satisfies $\alpha^2=s^2+4=r+4$. 

\begin{itemize}
  \item Case $\alpha \in \QD$. Combining this with \eqref{eq21}, we obtain $(r+4)\left(r^{2}+4 r+16\right) \in \QD^{2}$. Hence, there exists $\beta\in\QD$ such that $(r,\beta)\in C_1(\QD)$. Table \ref{points} shows that if $D\ne -1,\pm 3$ then $r\in\{0,-4\}$, if $D=-1$ then $r\in\{0,-4,-4\pm 4i\}$, if $D=-3$ then $r=\{0,-4,-8,-2\pm 2\sqrt{-3}\}$, and if $D=3$ then $r\in\{0,-4,4\pm 4\sqrt{3}\}$. The cases $r=0,-4,-2\pm 2\sqrt{-3}$ are not possible because they correspond to elementary arithmetic progressions (see Table \ref{elementary}). The cases $r=-4\pm 4i,-8,4\pm 4\sqrt{3}$ give $r+4=\pm 4i, -4,8\pm 4 \sqrt{3}$ respectively. These values are not squares in the corresponding quadratic fields.  

  \item Case $\alpha \notin \Q(\sqrt{D})$. Since $\alpha \in K=\Q(\sqrt{D})(s)$ and $[K:\Q(\sqrt{D})]=2$, we have
$\alpha = a+bs$ for some $a,b \in \Q(\sqrt{D})$. It follows that $\alpha^2 = a^2 + b^2 s^2 + 2ab s$. As $s^2 = r \in \Q(\sqrt{D})$ and $\alpha^2 = r+4 \in \Q(\sqrt{D})$, this implies that $ab=0$. If $b=0$, then $\alpha=a \in \Q(\sqrt{D})$, which contradicts our assumption.  
Therefore $b \neq 0$, and $\alpha=b s$. We conclude that $s\alpha = a s^2 = br \in \Q(\sqrt{D})$. Thus $r(r+4) \in$ $\Q(\sqrt{D})^{2}$. Combining this with \eqref{eq21}, we obtain $r(r+4)\left(r^{2}+4 r+16\right) \in \Q(\sqrt{D})^{2}$. Then there exists $\beta\in \QD$ such that $(r,\beta)\in C_3(\QD)$. Then Table \ref{points} shows that if $D\ne -1,\pm 3$ then $r\in\{0,-4\}$, if $D=-1$ then $r\in\{0,-4,-2\pm 2i\}$, if $D=-3$ then $r=\{0,-4,-2,-2\pm 2\sqrt{-3}\}$, and if $D=3$ then $r\in\{0,-4,-2\pm 2\sqrt{3}\}$. The cases $r=0,-4,-2\pm 2\sqrt{-3}$ are not possible because they correspond to elementary arithmetic progressions (see Table \ref{elementary}). The cases $r=-2\pm 2i,-2,-2+2\sqrt{3}$ give $r^2+4r+16=8,12,24$ respectively. All these values are not squares in their corresponding quadratic fields, in contradiction with \eqref{eq21}. 
\end{itemize}
Hence, $s \in \Q(\sqrt{D})$.

\section{Proof of Proposition \ref{prop}\,\eqref{prop3}}\label{tinQD}

Let \( t \in K \setminus \{0, \pm 1, \pm i, \pm 1 \pm \sqrt{2}\} \) corresponding to a non-elementary arithmetic progression, as described in ~\eqref{eq3}. We prove that $t \in \Q(\sqrt{D})$ under the hypothesis $\operatorname{rank}_{\Z}E_1^{\pm D}(\Q)=0.$

Assume $t \notin \Q(\sqrt{D})$. If $s=0$ then it follows from $t^{2}-s t-1=0$ that $t= \pm 1$, which is not possible. So $s \neq 0$. Hence, $t^{2}=s t+1 \notin \Q(\sqrt{D})$. Since $\left(t^{2}-1\right)^{2}=(s t)^{2}$, we obtain $t^{4}=\left(s^{2}+2\right) t^{2}-1$. Note that $s^{2}+2\ne 0$, otherwise $t=(s\pm \sqrt{2})/2$, $K=\Q(\sqrt{-2},i)$ and $t^4-2t^3+2t^2+2t+1=\pm 2(\sqrt{2}\pm i)$ is not an square in $K$, in contradiction with \eqref{eq6}.  Because $t^{2} \notin \Q(\sqrt{D})$, we have $t^{4}=\left(s^{2}+2\right) t^{2}-1 \notin \Q(\sqrt{D})$. Let $w=t^4$. Since
$$
t^{8}-\left(s^{4}+4 s^{2}+2\right) t^{4}+1 =\left(t^{2}-s t-1\right)\left(t^{2}+s t-1\right)\left(t^{4}+\left(s^{2}+2\right) t^{2}+1\right) =0,
$$
$w$ is a root of
\begin{equation}\label{eq22}
F(x)=x^{2}-\left(s^{4}+4s^{2}+2\right) x+1 \in \QD[x].
\end{equation}
Since $w \notin \QD, F(x)$ is irreducible in $\Q(\sqrt{D})[x]$ and $K=\Q(\sqrt{D},w)$. By \eqref{eq6},
$$
t^{8}+14 t^{4}+1=\left(t^{4}-2 t^{2}+2 t^{2}+2 t+1\right)\left(t^{4}+2 t^{2}+2 t^{2}-2 t+1\right) \in {K}^{2} .
$$
Therefore,
$$
w\left(w^{2}+14 w+1\right)=t^{4}\left(t^{8}+14 t^{4}+1\right) \in {K}^{2} .
$$
Hence, there exist $\alpha, \beta \in \Q(\sqrt{D})$ such that
$$
w\left(w^{2}+14 w+1\right)=\left(\alpha+\beta w\right)^{2}.
$$
Thus there exists $\omega_0\in\QD$ such that
\begin{equation*}
x\left(x^{2}+14 x+1\right)-\left(\alpha+\beta x\right)^{2}=F(x)\left(x-\omega_{0}\right) . 
\end{equation*}
Therefore, $\left(\omega_{0}, \alpha +\beta \omega_{0}\right)$ is a $\QD$-rational point on the elliptic curve $C_4\,:\,y^2=x\left(x^{2}+14 x+1\right)$. For any point $P=(x_0,y_0)\in C_4(\QD)$, we have
\begin{equation}\label{eqFbeta}
x\left(x^{2}+14 x+1\right)-(y_0-\beta x_0+\beta x)^{2}=F_{\beta}(x)\left(x-x_{0}\right) . 
\end{equation}

Let $(0,0)\in C_4(\QD)$. Then we obtain $F_{\beta}(x)=x^2+(14-\beta^2)x+1$. Comparing with \eqref{eq22} gives $14-\beta^{2}=-\left(s^{4}+4 s^{2}+2\right)$. So that $s^{4}+4 s^{2}+16=\beta^{2}$. Thus $\left(s, \beta\right)$ is a $\Q(\sqrt{D})$-rational point on the curve $C_5\,:\,y^2=x^{4}+4 x^{2}+16$. By Table \ref{points}, we obtain $s\in \{0,\pm 2,\pm 2i,\pm 1\pm \sqrt{-3}\}$, which corresponds to elementary arithmetic progressions (see Table \ref{elementary}).

 Let $(1,\pm 4)\in C_4(\QD)$. Then we obtain $F_{\beta}(x)=x^2+(15-\beta^2)x+(4+\beta)^2$. Comparing with \eqref{eq22} gives $(\beta,s)$ with $\beta=-3,-5$, and the minimal polynomial of $s$ is $x^4+4x^2+8$ and $x^4+4x^2-8$ respectively, which is not possible since $s\in \QD$.

Using the data in Table \ref{pointsC4}, we elaborate Table \ref{Pandbeta}, which shows for each $P\in C_4(\QD)$, $P\notin C_4(\Q)$,  solutions $\beta$ obtained from the equality $F_{\beta}(x)=F(x)$ by comparing \eqref{eq22} with \eqref{eqFbeta}. Note that replacing $P$ by $-P$ is equivalent to changing $\beta$ by $-\beta$. So only one element in the set $\{\pm P\}$ is shown.
\begin{center}
\renewcommand{\arraystretch}{1.3}
\begin{longtable}{|c|c|} 
\caption{} \label{Pandbeta} \\
\hline
 P & $\beta$ \\
\hline
 $(-3+2\sqrt{2},8-6\sqrt{2})$    &  $\pm i+(2\pm i)\sqrt{2}\notin \Q(\sqrt{2})$\\
 \hline
 $(-3-2\sqrt{2},+8+6\sqrt{2})$    &  $\pm i+(2\mp i)\sqrt{2}\notin \Q(\sqrt{2})$\\
 \hline
 $(-7\pm4\sqrt{3},0)$ & $\pm i(2+\sqrt{3})\notin\Q(\sqrt{3})$ \\
 \hline
\end{longtable}
\end{center}
Table \ref{Pandbeta} shows that $\beta\notin\QD$ in all the cases. This completes the proof that $t\in\QD$.

\section{Proof of Proposition \ref{prop}\,\eqref{prop4}}\label{tinQ}
Let {$D\ne \pm 2$} be a square-free integer satisfying $\operatorname{rank}_{\Z}E_1^{\pm D}(\Q)= 0$ and $\operatorname{rank}_{\Z}E_0^{D}(\Q)=0$. Let \( t \in K \setminus \{0, \pm 1, \pm i, \pm 1 \pm \sqrt{2}\} \) corresponding to a non-elementary arithmetic progression, as described in ~\eqref{eq3}. We prove that $t \in \Q$. Since $[K:\QD]=2$, there exists $\alpha \in K$ such that $\alpha^2 \in \QD$ and $K=\Q(\sqrt{D},\alpha)$. By \eqref{eq6}, there exist $\gamma_1,\delta_1,\gamma_2,\delta_2\in \QD$ such that 
\begin{equation}\label{Section6.E1}
\begin{cases}
 t^{4}-2 t^{3}+2 t^{2}+2 t+1&\!\!\!\!=\left(\delta_{1}+\gamma_{1}\alpha\right)^{2}, \\
 t^{4}+2 t^{3}+2 t^{2}-2 t+1&\!\!\!\!=\left(\delta_{2}+\gamma_{2} \alpha\right)^{2}.
\end{cases}
\end{equation}
Since $t \in \Q(\sqrt{D})$ and $\alpha \notin \QD$, we have $\gamma_{1} \delta_{1}=\gamma_{2} \delta_{2}=0$. There are four cases:

\begin{itemize}
  \item  $\gamma_{1}=\gamma_{2}=0$. Then
$$
t^{8}+14 t^{4}+1=\left(t^{4}-2 t^{3}+2 t^{2}+2 t+1\right)\left(t^{4}+2 t^{3}+2 t^{2}-2 t+1\right)=\left(\delta_{1} \delta_{2}\right)^{2} \in \Q(\sqrt{D})^{2} .
$$
So $\left(t^{2}, \delta_{1} \delta_{2}\right)$ is a $\Q(\sqrt{D})$-point on the curve $C_6\,:\,y^2=x^{4}+14 x^{2}+1$. By Table \ref{points}, we have \[C_6(\QD)=C_6(\Q)=\{(0,\pm 1),(\pm 1,\pm 4)\}\] for any integer $D$ such that $\operatorname{rank}_{\Z}E_1^{D}(\Q)=0$. Therefore, $t\in \{0,\pm 1\}$, which is not possible by hypothesis.  
\item  $\delta_{1}=\delta_{2}=0$. Then
$$
t^{8}+14 t^{4}+1=\left(t^{4}-2 t^{3}+2 t^{2}+2 t+1\right)\left(t^{4}+2 t^{3}+2 t^{2}-2 t+1\right)=\left(\alpha^2 \gamma_{1} \gamma_{2}\right)^{2} \in \Q(\sqrt{D})^{2}.
$$
So $\left(t^{2}, \alpha^2 \gamma_{1} \gamma_{2}\right)\in C_6(\QD)$, which is not possible by  a similar argument to the case $\gamma_{1}=\gamma_{2}=0$.

\item  $\gamma_{1}=\delta_{2}=0$. The first equation of \eqref{Section6.E1} shows that $t^{4}-2 t^{3}+2 t^{2}+2 t+1=\delta_1^2$. So $\left(t, \delta_{1}\right)$ is a $\Q(\sqrt{D})$-point on the curve $C_0\,:\ y^2=x^{4}-2 x^{3}+2 x^{2}+2 x+1$. By Lemma \ref{newlemma}, we have $C_0(\QD)=C_0(\Q)$ for any square-free integer $D\ne \pm 2$ such that $\operatorname{rank}_{\Z}E_0^{D}(\Q)=0$. Then $t\in\Q$.
\item  $\delta_{1}=\gamma_{2}=0$. The second equation of \eqref{Section6.E1} shows that $t^{4}+2 t^{3}+2 t^{2}-2 t+1=\delta_2^2$. So $\left(-t, \delta_{2}\right)$ is a $\Q(\sqrt{D})$-point on the curve $C_0$. Similar to the case $\gamma_{1}=\delta_{2}=0$, we obtain that $t\in\Q$ if $D\ne \pm 2$.
\end{itemize}
So if $D\ne \pm 2$ then $t\in \Q$. 

\section{Proof of Theorem \ref{main} and \ref{maincor}}

Let $D$ be an square-free integer satisfying $\operatorname{rank}_{\Z}E_1^{\pm D}(\Q)= 0$ and $K$ be a quadratic extension of $\QD$. Let $(a^2,b^2,c^2,d^2,e^2)$ be an arithmetic progression of five squares properly defined over $K$. If $(a^2,b^2,c^2,d^2,e^2)$ is elementary but non-constant, then it is equivalent to one of the following
\begin{itemize}
    \item $(-2,-1,0,1,2)$ and $K=\Q(i,\sqrt{2})$. In particular, $D=-1,\pm 2$. 
    \item $(0,1,2,3,4)$ and $K=\Q(\sqrt{3},\sqrt{2})$. In particular, $D=2,3$. Note that the arithmetic progression $(4,3,2,1,0)$ is equivalent to $(0,1,2,3,4)$.
\end{itemize}
Now suppose that $(a^2,b^2,c^2,d^2,e^2)$  is not an elementary arithmetic progression. By Proposition \ref{prop} \eqref{prop3}, there exists $t\in\QD$ such that 
 $(a^2,b^2,c^2,d^2,e^2)$ is equivalent to the arithmetic progression
 $$
((t^{2}-2 t-1)^2, \beta^2, (t^{2}+1)^2, \gamma^2,(t^{2}+2 t-1)^2),
 $$
where $\beta^2=G(t)$ and $\gamma^2=G(-t)$, and $G(x)=x^4-2x^3+2x^2+2x+1$. \\

Let us prove Theorem \ref{main}:
\begin{itemize}
\item[(A)] Assume $\operatorname{rank}_{\Z}E_0^{D}(\Q)=0$. Since $D\ne \pm 2$. By Proposition \ref{prop} \eqref{prop4}, we have $t\in \Q$. In this case we have \( C_0(\mathbb{Q}(\sqrt{D})) = C_0(\mathbb{Q}) \). Hence, if the value of \( G(t) \) is a square over \( \mathbb{Q}(\sqrt{D}) \), it is already a square over \( \mathbb{Q} \). Consequently, four out of the five terms of the arithmetic progression are defined over \( \mathbb{Q} \), and therefore the full arithmetic progression of five terms is properly defined over a quadratic extension of \( \mathbb{Q} \).
\item[(B)] Assume $\operatorname{rank}_{\Z}E_0^{D}(\Q)\ne 0$ and that the class number of $\QD$ is $1$. By Proposition \ref{prop} \eqref{prop3}, we have $t\in \QD$. Since the ring of integers of $\QD$ is a unique factorization domain, there exist $\alpha,\delta\in \Q(\sqrt{D})$ such that $G(-t)=\alpha\,\delta^2$. We can conclude that \(\alpha\) is not a square in \(\mathbb{Q}(\sqrt{D})\), since otherwise there would exist a nontrivial arithmetic progression of five squares in \(\mathbb{Q}(\sqrt{D})\). However, this is impossible because \cite[Proposition~5.2]{GJX} states that a necessary condition is that \(\operatorname{rank} E_1^{\pm D}(\mathbb{Q}) \geq 2\). Then the arithmetic progression $(a^2,b^2,c^2,d^2,e^2)$ is properly defined over $K=\Q(\sqrt{D},\omega)$, where $\omega^2=\alpha$.
\end{itemize}
Now, the proof of Theorem \ref{maincor} assuming that $(a^2,b^2,c^2,d^2,e^2)$ is not an elementary arithmetic progression:
\begin{itemize}
\item $D\in\{-1,3\}$: The proof is analogous to the case $(A)$ since $\operatorname{rank}_{\Z}E_0^{D}(\Q)=0$ and {$D\ne \pm 2$}.
\item $D= \pm 2$: In this case $\operatorname{rank}_{\Z}E_0^{D}(\Q)=0$. By Proposition \ref{prop} \eqref{prop3} we have $t\in \Q(\sqrt{D})$. The proof is analogous to the case $(B)$ since the class number of $\Q(\sqrt{D})$ is $1$.
\end{itemize}

\section{Proof of Theorem \ref{six}}
We can view Theorem \ref{six} as a corollary of Theorem \ref{main} and \ref{maincor}. Let $D$ be an square-free integer satisfying $\operatorname{rank}_{\Z}E_1^{\pm D}(\Q)= 0$ and $K$ a quadratic extension of $\QD$. We will prove that no arithmetic progression of 5 squares properly defined over $K$ can be extended to one of length $6$.

Let $(a^2,b^2,c^2,d^2,e^2,f^2)$ be an arithmetic progression of six squares properly defined over $K$. Then $(a^2,b^2,c^2,d^2,e^2)$ or  $(b^2,c^2,d^2,e^2,f^2)$ is an arithmetic progressions of five squares properly defined over $K$. Assume $D\ne -1,\pm 2,3$. We can apply Theorem \ref{main}:
\begin{itemize}
\item[(A)] Assume $\operatorname{rank}_{\Z}E_0^{D}(\Q)=0$, then there does not exist any non-constant arithmetic progression of six squares properly defined over $K$, since there is not of length five.
\item[(B)] Assume $\operatorname{rank}_{\Z}E_0^{D}(\Q)\ne 0$ and that the class number of $\QD$ is $1$. Suppose that $(a^2,b^2,c^2,$ $d^2,e^2)$ is an arithmetic progressions of five squares properly defined over $K$ (the other case is completely equivalent). Then there exist $x_1,x_2,x_3,x_4,x_5,\alpha \in\Q(\sqrt{D})$ and $\alpha$ is non-square such that $(a^2,b^2,c^2,d^2,e^2)$ is equivalent to $(x_1^2,x_2^2,x_3^2,$ $\alpha \, x_4^2,x_5^2)$. Suppose that $s\in \QD$ satisfies
$$
(sa^2,sb^2,sc^2,sd^2,se^2)=(x_1^2,x_2^2,x_3^2, \alpha \, x_4^2,x_5^2).
$$
From the equalities $sa^2=x_1^2$ and $sd^2=\alpha \, x_4^2$ we obtain $\alpha\in\QD^2$ which is a contradiction.
\end{itemize}
For the cases $D=-1,3,\pm 2$, we apply Theorem \ref{maincor}:
\begin{itemize}
    \item If $D=-1$, then $K=\Q(i,\sqrt{2})$ and the unique non-constant arithmetic progression of five squares properly defined over $K$ is, up to equivalence, $(-2,-1, 0,1,2)$. Therefore there is only two ways to extend to an arithmetic progression of length $6$:
$$
(-3,-2,-1,0,1,2)\qquad\text{and}\qquad (-2,-1,0,1,2,3).
$$
However, both possibilities are impossible, as they would imply that $\sqrt{-3}\in K$ or $\sqrt{-3}\in K$.   
    \item If $D=3$, then $K=\Q(\sqrt{3},\sqrt{2})$ and the unique non-constant arithmetic progression of five squares properly defined over $K$ is, up to equivalence, $(0,1,2,3,4)$. A similar argument to the previous case leads us to $i\in K$ or $\sqrt{5}\in K$, which again results in a contradiction.
    \item If $D=2$, then there are three possibilities. The first two are when $K=\Q(\sqrt{2},\sqrt{3})$ and the arithmetic progression of five squares is $(0,1,2,3,4)$, or $K=\Q(\sqrt{2},i)$ with $(-2,-1, 0,1,2)$. Then the proofs is again as the one given in the the cases $D=-1$ and $D=3$ respectively. Finally, the last case is when $K=\Q(\sqrt{2},\sqrt{\alpha})$ with $\left(a^2,b^2,c^2,\alpha \, d^2,e^2\right)$ where $a,b,c,d,e,\alpha \in\Q(\sqrt{2})$ and $\alpha$ is non-square. In this case, the proof proceeds in the same way as in the case $D \neq -1, \pm 2, 3$, by applying Theorem \ref{main} (B)    
  \item If $D=-2$, then $K=\Q(i,\sqrt{-2})$ and the arithmetic progression of five squares is equivalent to $(-2,-1, 0,1,2)$ and the proof is equal to the case $D=-1$, or $K=\Q(\sqrt{-2},\sqrt{\alpha})$ with $\left(a^2,b^2,c^2,\alpha \, d^2,e^2\right)$ where $a,b,c,d,e,\alpha \in\Q(\sqrt{-2})$ and $\alpha$ is non-square. In this situation, the argument can be carried out in the same manner as for the case $D \neq -1, \pm 2, 3$, making use of Theorem \ref{main} (B).
\end{itemize}

\section{Parametrization in the cases $\Q(\sqrt{-2})$ and $\Q(\sqrt{2})$}\label{sect_para}
Let \(D\) be a square-free integer such that \(\operatorname{rank}_{\mathbb{Z}} E_1^{\pm D}(\mathbb{Q}) = 0\). We have proved that any non-elementary arithmetic progression of five squares over a quadratic extension of \(\QD\) is equivalent to 
\[
\bigl((t^{2}-2t-1)^2,\, \beta^2,\, (t^{2}+1)^2,\, \alpha\delta^2,\, (t^{2}+2t-1)^2 \bigr),
\]
where \(t, \alpha, \beta, \delta \in \mathbb{Q}(\sqrt{D})\) satisfy \(\beta^2 = G(t)\) and \(\alpha\delta^2 = G(-t)\). In particular, for any non-elementary arithmetic progression of five squares, there exists a point \(R \in E_0(\mathbb{Q}(\sqrt{D}))\), and conversely. Note that this correspondence is not one-to-one. This construction recovers the one developed in \cite{GJX} for the special case $D=1$. That is, for any \(R \in E_0(\mathbb{Q})\) we obtain an arithmetic progression of five squares over a quadratic field, except when \(R \in E_0(\mathbb{Q})_{\operatorname{tors}}\), in which case the progression is trivial. Note that \(\operatorname{rank}_{\mathbb{Z}} E_1^{\pm 1}(\mathbb{Q}) = 0\). Let $T_1=(-3,0)$, $T_2=(1,0)$ and $P=(0,3)$, then 
$$
E_0(\Q)=\langle T_1 \rangle \oplus \langle T_2 \rangle \oplus \langle P \rangle.
$$
If $n\in\Z$, $n\ge 0$, the set of points $S_n=\{n_1T_1+n_2T_2+m P \,:\, n_1,n_2\in\{0,1\} \,,\,m\in\{n,-n-1\}\}$ corresponds to the same arithmetic progression of five squares, up to equivalence. For example, if $n=1$ the set $S_1$ corresponds to the arithmetic progression $(7^2, 13^2, 17^2, (\sqrt{409})^2, 23^2)$.

If $D\ne \pm 2$ such that $\operatorname{rank}_{\Z}E_0^{D}(\Q)=0$, then $E_0(\QD)=E_0(\Q)$. Then any point $R$ in $E_0(\QD)$ corresponds to an arithmetic progression of five squares over a quadratic extension of \(\Q\). Therefore not properly defined over a quadratic extension of $\QD$.

Let us study briefly the cases $D=-2$ and $D=2$. First, $D=-2$. We have 
$$
E_0(\Q(\sqrt{-2}))=\langle T_1 \rangle \oplus \langle T'_2 \rangle \oplus \langle P \rangle,
$$
where $T'_2=(1-2\sqrt{-2},4+4\sqrt{-2})$ satisfies $2T'_2=T_2$. In this case, if $n\in\Z$, $n\ge 0$, the set of points $\{n_1 T_1+n_2T'_2+m P \,:\, n_1\in\{0,1\}\,,\,n_2\in\{1,3\}\,,\,m\in\{n,-n-1\}\}$ corresponds to the same arithmetic progression of five squares, up to equivalence. Table \ref{example} shows the corresponding arithmetic progressions of five squares for $n=0,1,2,3$. 
\begin{center}
\begin{longtblr} 
[caption = {Examples of arithmetic progressions.}, label=example]
{cells = {mode=imath},hlines,vlines,colspec  = cc}
n & \mbox{arithmetic progression}\\
0 & (2^2,1^2,-2,-5,-2\cdot 2^2)\\
1 & (34^2,23^2,-2\cdot 7^2,-29\cdot 5^2, -2\cdot 26^2)\\
2 & (5986^2,647^2,-2\cdot 4183^2,-70408565,-2\cdot 7274^2)\\
3 & (25953218^2,18240049^2,-2\cdot 2021231^2,-349040886543845, -2\cdot 18572978^2)\\
\end{longtblr}
\end{center}
These arithmetic progressions are of the form $\left(a^2,b^2,-2c^2,-m d^2,-2e^2\right)$, where $a,b,c,d,e,m\in\Z$ and $m>0$ is square-free. In particular, they are properly defined over $K=\Q(\sqrt{-2},\sqrt{-m})$. We have checked the above assertion for $0\leq n\le 11$. These computations support Conjecture \ref{Conj} (i).

In the case $D=2$, we have 
$$
E_0(\Q(\sqrt{2}))=\langle T_1 \rangle \oplus \langle T_2 \rangle \oplus \langle P' \rangle,
$$ 
where $P'=(1+2\sqrt{2},4)$ satisfies $2P'+T_1=-P$. In this case, if $n\in\Z$, $n>0$ odd, the set of points $S_n=\{n_1T_1+n_2T_2+m P' \,:\, n_1,n_2\in\{0,1\} \,,\,m\in\{n,-n-2\}\}$ corresponds to the same arithmetic progression of five squares, up to equivalence. Table \ref{example} shows the corresponding arithmetic progressions of five squares for $n=1,3,5$. 
\begin{center}
\begin{longtblr} 
[caption = {Examples of arithmetic progressions.}, label=example]
{cells = {mode=imath},hlines,vlines,colspec  = cc}
n & \mbox{arithmetic progression}\\
1 & (2\cdot 4^2,5^2,2\cdot 3^2,11, 2^2) \\
3 & (2\cdot 120^2,241^2,2\cdot 209^2,3\cdot 5\cdot 659, 382^2) \\
5 & (2\cdot 142324^2,221285^2,2\cdot 169443^2,11\cdot 10691\cdot 560171, 272638^2) \\
\end{longtblr}
\end{center}
In these cases, the arithmetic progressions are of the form $\left(2 a^2,b^2,2c^2,m,e^2\right)$, where $a,b,c,d,e,m\in\Z$ and $m>0$ is square-free. In particular, they are properly defined over $K=\Q(\sqrt{2},\sqrt{m})$. We have checked the above assertion for $0\leq n\le 21$ odd. These computations support Conjecture \ref{Conj} (ii).


\newpage
\begin{thebibliography}{99}
\bibitem{art1} A. Aigner, \"Uber die M\"oglichkeit von $x^4 + y^4 = z^4$ in quadratische
K\"orper, J. Math. Verein. {\bf43} (1934), 226-228.

\bibitem{Magma} W. Bosma, J. Cannon, and C. Playoust,
The \texttt{Magma} algebra system. I. The user language,
 J. Symbolic Comput. {\bf 24}(1997), no. 3-4, 235-265. 
 
\bibitem{art5} A. Bremner and S. Siksek, Squares in arithmetic progression over cubic fields,
Int. J. Number Theory {\bf 12} (2016), no. 5, 1409-1414.

\bibitem{art7} D. K. Faddeev, Group of divisor classes on the curve defined by the equation
$x^4 + y^4 = 1$,  Soviet Math. Dokl. {\bf 1} (1960), 1149-1151;  Dokl. Akad. Nauk SSSR {\bf 134}
(1960), 776-777 (Russian original). 

\bibitem{GJN1}
E. Gonz\'alez-Jim\'enez and F. Najman, Growth of torsion of  elliptic curves upon base change, Math. Comp. \textbf{89} (2020), 1457-1485.

\bibitem{GJN2}
E. Gonz\'alez-Jim\'enez and F. Najman, An algorithm for determining torsion growth of elliptic curves, Exp. Math. {\bf 32} (2023), 70-81.

\bibitem{GJT1}
E. Gonz\'alez-Jim\'enez and J.M. Tornero, Torsion of rational elliptic curves over quadratic fields, Rev. R. Acad. Cienc. Exactas F\'is. Nat. Ser. A Math. RACSAM {\bf 108} (2014), 923-934.

\bibitem{GJT2}
E. Gonz\'alez-Jim\'enez and J.M. Tornero, Torsion of rational elliptic curves over quadratic fields II, Rev. R. Acad. Cienc. Exactas F\'is. Nat. Ser. A Math. RACSAM {\bf 110} (2016), 121-143.

\bibitem{GJX} E. Gonz\'alez-Jim\'enez and X. Xarles,
Five squares in arithmetic progression over quadratic fields,
Rev. Mat. Iberoam. {\bf 29} (2013), no. 4, 1211-1238.


\bibitem{Kra} K. Kramer,
Arithmetic of elliptic curves upon quadratic extension,
Trans. Am. Math. Soc. {\bf 264} (1981) 121-135.
 
\bibitem{art12} L. J. Mordell, The Diophantine equation $x^4+y^4 = 1$ in algebraic number fields, Acta
Arith. {\bf 14} (1967/1968), 347-355. 

\bibitem{N}
F. Najman, Torsion of rational elliptic curves over cubic fields and sporadic points on $X_1(N)$,  Math. Res. Letters, {\bf 23} (2016) 245-272.

\bibitem{lmfdb} The LMFDB Collaboration, The L-functions and modular forms database, \url{https://www.lmfdb.org}, 2024, [Online; accessed 15 April 2024].

\bibitem{art17} X. Xarles,
Squares in arithmetic progression over number fields,
J. Number Theory {\bf 132} (2012), no. 3, 379-389.

\bibitem{mathematica} Wolfram Research, Inc.,
Mathematica, Version 12.1, Champaign, IL (2020). 

\end{thebibliography}
\end{document}